\documentclass[10pt]{amsart}

\usepackage{amsmath}
\usepackage{amsthm}
\usepackage{amssymb}

\usepackage[all]{xy}

\newcommand{\dual}[1]{{#1}^\vee}
\newcommand{\set}[2]{\left\{ #1 \mid #2 \right\}}
\newcommand{\defn}{\emph}
\newcommand{\gen}[1]{\langle #1 \rangle}
\renewcommand{\c}[1]{{#1}^c}

\newtheorem{theorem}{Theorem}
\newtheorem{proposition}[theorem]{Proposition}

\theoremstyle{definition}
\newtheorem{definition}[theorem]{Definition}
\newtheorem{definition/lemma}[theorem]{Definition/Lemma}
\newtheorem{remark}[theorem]{Remark}
\newtheorem{example}[theorem]{Example}

\newtheorem{problem}[theorem]{Problem}

\begin{document}

\title{Shellability and the strong gcd-condition}
\author{Alexander Berglund}
\address{Department Of Mathematics\\
Stockholm University\\
SE-106 91 Stockholm\\
Sweden}
\email{alexb@math.su.se}

\begin{abstract}
Shellability is a well-known combinatorial criterion for verifying that a simplicial complex is Cohen-Macaulay. Another notion familiar to commutative algebraists, but which has not received as much attention from combinatorialists as the Cohen-Macaulay property, is the notion of a \emph{Golod ring}. Recently in \cite{BJ}, a criterion on simplicial complexes reminiscent of shellability, called the \defn{strong gcd-condition}, was shown to imply Golodness of the associated Stanley-Reisner ring. The two algebraic notions were tied together by Herzog, Reiner and Welker, \cite{HRW}, who showed that if the Alexander dual $\Delta^\vee$ is sequentially Cohen-Macaulay then $\Delta$ is Golod. In this paper, we present a combinatorial companion of this result, namely that if $\Delta^\vee$ is (non-pure) shellable then $\Delta$ satisfies the strong gcd-condition. Moreover, we show that all implications just mentioned are strict in general but that they are equivalences if $\Delta$ is a flag complex.
\end{abstract}

\maketitle

\section*{Introduction}
Let $\Delta$ be a finite simplicial complex with vertex set $V = \{v_1,\ldots,v_n\}$. We assume that there are no `ghost vertices', that is, we assume that $\{v_i\}\in \Delta$ for all $i$. Let $k$ be a field. Recall that the \defn{Stanley-Reisner ring} associated to $\Delta$ is the quotient
$$k[\Delta] = k[x_1,\ldots,x_n]/I_\Delta,$$
where $I_\Delta$ is the ideal in the polynomial ring $k[x_1,\ldots,x_n]$ generated by the monomials $x_{i_1}\ldots x_{i_r}$ for which $\{v_{i_1},\ldots,v_{i_r}\}\not\in \Delta$. We will say that $\Delta$ is sequentially Cohen-Macaulay, or Golod, if the Stanley-Reisner ring $k[\Delta]$ has that property. For the definitions of these notions from commutative algebra, see \cite{stanley} Definition III.2.9 and \cite{Avramov2} p.~ 42, respectively. Roughly speaking, the ring $k[\Delta]$ is Golod if the ranks of the modules in the minimal free resolution of the field $k$, viewed as a module over $k[\Delta]$, have the fastest possible growth. There are several equivalent characterizations of Golod rings, see for instance Sections 5.2 and 10.3 in \cite{Avramov2} and the references therein. For Stanley-Reisner rings one can obtain even nicer characterizations, see \cite{B}, \cite{BBH}, \cite{BJ}.

In this paper, we will be interested in the combinatorial companions of these notions: shellability and the strong gcd-condition. Let us begin by recalling their definitions. If $F_1,\ldots,F_r \subseteq V$ then let
$$\gen{F_1,\ldots,F_r}$$
denote the simplicial complex generated by $F_1,\ldots,F_r$. It consists of all subsets $F\subseteq V$ such that $F\subseteq F_i$ for some $i$.
\begin{definition}[Bj\"orner, Wachs \cite{BW}] \label{def:shellable}
A simplicial complex $\Delta$ is called \defn{shellable} if the facets of $\Delta$ admit a \defn{shelling order}. A shelling order is a linear order, $F_1,\ldots, F_r$, of the facets of $\Delta$ such that for $2\leq i \leq r$, the simplicial complex
$$\gen{F_i}\cap \gen{F_1,\ldots, F_{i-1}}$$
is pure of dimension $\dim(F_i) - 1$.
\end{definition}
What we call shellability here is sometimes referred to as \emph{non-pure} shellability since it is not assumed that the complex $\Delta$ is pure. As is well known and widely exploited, shellability is a combinatorial criterion for verifying that a pure complex is Cohen-Macaulay. The notion of sequentially Cohen-Macaulay complexes, due to Stanley, was conceived as a non-pure generalization of the notion of Cohen-Macaulay complexes that would make the following proposition true:

\begin{proposition}[Stanley, \cite{stanley}] \label{prop:shellable-CM}
A shellable simplicial complex is sequentially Cohen-Macaulay.
\end{proposition}

We now move to the strong gcd-condition.
\begin{definition}[J\"ollenbeck \cite{jol}] \label{def:sgcd}
A simplicial complex $\Delta$ is said to satisfy the \defn{strong gcd-condition} if the set of minimal non-faces of $\Delta$ admits a \defn{strong gcd-order}. A strong gcd-order is a linear order, $M_1,\ldots,M_r$, of the minimal non-faces of $\Delta$ such that whenever $1\leq i<j \leq r$ and $M_i\cap M_j = \emptyset$, there is a $k$ with $i<k\ne j$ such that $M_k \subseteq M_i\cup M_j$.
\end{definition}

The strong gcd-condition was introduced because of its relation to the Golod property. In \cite{jol}, J\"ollenbeck made a conjecture a consequence of which was that the strong gcd-condition is sufficient for verifying that a complex is Golod. One of the main results of the paper \cite{BJ} was a proof of that conjecture, thus establishing the truth of the next proposition.

\begin{proposition}[Berglund, J\"ollenbeck \cite{BJ}] \label{prop:sgcd-Golod}
A simplicial complex satisfying the strong gcd-condition is Golod.
\end{proposition}

The following result ties together the notions of sequentially Cohen-Macaulay rings and Golod rings, via the \emph{Alexander dual}. Recall that the Alexander dual of $\Delta$ is the simplicial complex
$$\dual{\Delta} = \set{F\subseteq V}{\c{F} \not\in \Delta}.$$
Here and henceforth $\c{F}$ denotes the complement of $F$ in $V$. The facets of $\dual{\Delta}$ are the complements in $V$ of the minimal non-faces of $\Delta$.

\begin{proposition}[Herzog, Reiner, Welker \cite{HRW}] \label{prop:CM-Golod}
If the Alexander dual $\dual{\Delta}$ is sequentially Cohen-Macaulay, then $\Delta$ is Golod.
\end{proposition}

What we have said so far can be summarized by the following diagram of implications:
$$\xymatrix{\mbox{$\dual{\Delta}$ shellable} \ar@{==>}[r] \ar@{=>}[d] & \mbox{$\Delta$ strong gcd} \ar@{=>}[d] \\ \mbox{$\dual{\Delta}$ seq.\ CM} \ar@{=>}[r] & \mbox{$\Delta$ Golod}}$$

This diagram seems to indicate that the strong gcd-condition plays the same role for the Golod property as shellability does for the property of being sequentially Cohen-Macaulay. What we wish to do next is to tie together the accompanying combinatorial notions by proving the implication represented by the dashed arrow. After that, we will give examples of simplicial complexes, $\Delta_1$, $\Delta_2$ and $\Delta_3$, having the following configurations of truth values in the diagram: \\
\begin{center}
\begin{tabular}{|c|c|c|}
\hline
$\Delta_1$ & $\Delta_2$ & $\Delta_3$ \\
\hline
$\begin{array}{cc} F & T \\ F & T \end{array}$ & $\begin{array}{cc} F & T \\ T & T \end{array}$ & $\begin{array}{cc} F & F \\ F & T \end{array}$\\
\hline
\end{tabular}
\end{center}
\vspace{10pt}
\noindent In particular, all implications in the diagram are strict. However, we will finish by proving that if $\Delta$ is a flag complex, then all arrows are in fact equivalences.

\section*{Weak shellability}

\begin{proposition} \label{prop:shellable-sgcd}
If $\dual{\Delta}$ is shellable then $\Delta$ satisfies the strong gcd-condition.
\end{proposition}

\begin{proof}
Let $F_1,\ldots, F_r$ be a shelling order of the facets of $\dual{\Delta}$. The minimal non-faces of $\Delta$ are then $\c{F_1},\ldots, \c{F_r}$. We claim that the reversed order, $\c{F_r},\ldots,\c{F_1}$, is a strong gcd-order for $\Delta$. By the standing assumption that $\Delta$ has no ghost-vertices, $|\c{F_i}| \geq 2$, or in other words $|F_i|\leq |V|-2$, for all $i$.

Let $1\leq i<j \leq r$ and suppose that $\c{F_i} \cap \c{F_j} = \emptyset$. We must produce a $k$ with $i\ne k < j$ such that $\c{F_k}\subseteq \c{F_i}\cup\c{F_j}$. The assumption means that $F_i\cup F_j = V$. Combining this with the fact $|F_i|\leq |V|-2$, we get
$$|F_i\cap F_j| \leq |F_j|-2.$$
Since $F_1,\ldots, F_r$ is a shelling order, the complex
$$\gen{F_j}\cap\gen{F_1,\ldots,F_{j-1}}$$
is pure of dimension $\dim(F_j)-1$. Of course, $F_i\cap F_j$ is contained in this complex. Let $H$ be a facet of the complex containing $F_i\cap F_j$. Then $|H| = |F_j|-1$. If $H\subseteq F_i$, then $H\subseteq F_i\cap F_j$, but this is impossible since $|F_i\cap F_j|\leq |F_j|-2$. Therefore, $H$ is contained in some $F_k$ where $i\ne k < j$. Hence, $F_i\cap F_j\subseteq H \subseteq F_k$, which implies that $\c{F_k} \subseteq \c{F_i}\cup\c{F_j}$. This finishes the proof.
\end{proof}

By using the correspondence between minimal non-faces of $\Delta$ and facets of $\dual{\Delta}$, one can rephrase the strong gcd-condition as a property of $\dual{\Delta}$ in the following way:

\begin{definition} \label{def:weakly shellable}
A simplicial complex $\Delta$ is called \defn{weakly shellable} if the facets of $\Delta$ admit a \emph{weak shelling order}. A weak shelling order is a linear order, $F_1,\ldots,F_r$ of the facets of $\Delta$ such that if $1\leq i<j\leq r$ and $F_i\cup F_j = V$ then there is a $k$ with $i\ne k < j$ such that $F_i\cap F_j\subseteq F_k$.
\end{definition}
Then the following is clear by definition:
\begin{proposition} \label{prop:ws-sgcd}
Let $\Delta$ be a simplicial complex and let $M_1,\ldots, M_r$ be its minimal non-faces. Then the facets of $\dual{\Delta}$ are $F_i = \c{M_i}$, $i=1,\ldots,r$, and the order $M_1,\ldots,M_r$ is a strong gcd-order if and only if $F_r,F_{r-1},\ldots,F_1$ is a weak shelling order.
\end{proposition}
In fact, the proof of Proposition \ref{prop:shellable-sgcd} shows the following:
\begin{proposition} \label{prop:justification}
Let $\Delta$ be a simplicial complex such that $|F|\leq |V|-2$ for all $F\in \Delta$. Then any shelling order of the facets of $\Delta$ is a weak shelling order.
\end{proposition}

\begin{remark}
Note that if $\Delta$ is a $d$-dimensional simplicial complex with $|V|\geq 2d + 3$, then $\Delta$ is automatically weakly shellable because in this case $|F\cup G|<|V|$ for all faces $F,G\in \Delta$.
\end{remark}

\section*{Examples}
\begin{example}
Let $\Delta_1$ be the simplicial complex with vertex set $\{1,2,3,4,5,6\}$ and minimal non-faces $\{1,2,3\}$, $\{1,2,6\}$, $\{4,5,6\}$. The Alexander dual $\Delta_1^\vee$ has facets $\{1,2,3\}$, $\{3,4,5\}$, $\{4,5,6\}$, and it is not Cohen-Macaulay because the link of the vertex $3$ is one-dimensional but not connected. However, the order in which the minimal non-faces of $\Delta_1$ appear above is in fact a strong gcd-order.
\end{example}

\begin{example}
Let $\Delta_2^\vee$ be the triangulation of the `dunce hat' with vertices $1,2,\ldots,8$ and facets
\begin{align*}
\{1,2,4\}, \{1,2,7\}, \{1,2,8\}, \{1,3,4\}, \{1,3,5\}, \{1,3,6\}, \{1,5,6\}, \{1,7,8\}, \{2,3,5\}, \\
\{2,3,7\}, \{2,3,8\}, \{2,4,5\}, \{3,4,8\}, \{3,6,7\}, \{4,5,6\}, \{4,6,8\}, \{6,7,8\}.
\end{align*}
It is well known that any triangulation of the dunce hat is Cohen-Macaulay but not shellable. Furthermore, for this particular triangulation, $|V| = 8\geq 7 = 2\dim(\Delta_2^\vee) + 3$, so $\Delta_2^\vee$ is automatically weakly shellable, which means that $\Delta_2$ satisfies the strong gcd-condition.
\end{example}

\begin{example}
Let $\Delta_3$ be the simplicial complex with vertices $0,1,\ldots,9$ and minimal non-faces
$$\{0,1,5,6\},\{1,2,6,7\},\{2,3,7,8\},\{3,4,8,9\},\{0,4,5,9\},\{5,6,7,8,9\}.$$
One can check by a direct computation that this simplicial complex is Golod. However, the strong gcd-condition is violated.

Next, if the dual complex $\Delta_3^\vee$ were sequentially Cohen-Macaulay, then by \cite{stanley} Proposition III.2.10 the pure subcomplex $\Gamma$ generated by the facets of maximum dimension would be Cohen-Macaulay. The dual $\Gamma^\vee$ of this subcomplex has minimal non-faces
$$\{0,1,5,6\},\{1,2,6,7\},\{2,3,7,8\},\{3,4,8,9\},\{0,4,5,9\}.$$
We will argue that $\Gamma^\vee$ is not Golod, so that, by Proposition \ref{prop:CM-Golod}, $\Gamma$ cannot be Cohen-Macaulay. By \cite{GPW} Theorem 3.5 (b), the complex $\Gamma^\vee$ is Golod if and only if the complex $K$ with vertices $0,1,2,3,4$ and minimal non-faces
$$\{0,1\},\{1,2\},\{2,3\},\{3,4\},\{0,4\},$$
is Golod. $K$ is a triangulation of $S^1$ and is therefore Gorenstein$^*$. However, being Golod and Gorenstein$^*$ are mutually exclusive properties as soon as there are at least two minimal non-faces, see \cite{BJ}, so $K$ is not Golod.
\end{example}

The reader might wonder why we have not provided an example with the table
$$
\begin{array}{cc}
F & F \\ T & T
\end{array}
$$
The Alexander dual of a simplicial complex having this table would need to be a non-shellable sequentially Cohen-Macaulay complex with $|V| < 2d+3$. Already finding complexes meeting these specifications seems difficult: All but one of the examples of non-shellable Cohen-Macaulay complexes found in \cite{hachimori} satisfy $|V|\geq 2d+3$, and are therefore weakly shellable for trivial reasons. The exception is the classical $6$-vertex triangulation of the real projective plane, which is however easily seen to be weakly shellable. Also, Gr\"abe's example \cite{grabe} of a complex which is Gorenstein when the characteristic of the field $k$ is different from $2$ but not Gorenstein otherwise is weakly shellable. It has been shown that all 3-balls with fewer than $9$ vertices are extendably shellable, and that all 3-spheres with fewer than $10$ vertices are shellable, see \cite{lutz}, so there is no hope in finding an example there. The author would however be very surprised if no example existed.
\begin{problem}
Find a sequentially Cohen-Macaulay complex which is not weakly shellable.
\end{problem}

\section*{Flag complexes}
Recall that a flag complex is a simplicial complex all of whose minimal non-faces have two elements. Order complexes associated to partially ordered sets are important examples flag complexes. Note that the Alexander dual of a flag complex is pure, and for pure complexes sequentially Cohen-Macaulay means simply Cohen-Macaulay.

\begin{proposition} \label{prop:flag-converse}
Suppose that $\Delta$ is a flag complex. Then the following are equivalent:
\begin{enumerate}
\item $\dual{\Delta}$ is shellable. \label{shellable}
\item $\Delta$ satisfies the strong gcd-condition. \label{sgcd}
\item $\dual{\Delta}$ is Cohen-Macaulay. \label{CM}
\item $\Delta$ is Golod. \label{Golod}
\end{enumerate}
\end{proposition}

\begin{proof}
For the equivalence of \eqref{sgcd}, \eqref{CM} and \eqref{Golod}, see \cite{BJ} Theorem 4. The implication \eqref{shellable} $\Rightarrow$ \eqref{sgcd} follows from Proposition \ref{prop:shellable-sgcd}. What remains to be verified is the implication \eqref{sgcd} $\Rightarrow$ \eqref{shellable} and this is contained in the next proposition.
\end{proof}

\begin{proposition} \label{prop:weak-flag}
If $\Delta$ is a flag complex then any weak shelling order of the facets of $\dual{\Delta}$ is a shelling order.
\end{proposition}

\begin{proof}
Let $F_1,\ldots,F_r$ be a weak shelling order of the facets of $\dual{\Delta}$. The complements $\c{F_1},\ldots,\c{F_r}$ are the minimal non-faces of the flag complex $\Delta$, so $|\c{F_i}| = 2$ and $|F_i| = |V|-2$ for all $i$. Let $j\geq 2$ and consider the complex
$$\gen{F_j}\cap\gen{F_1,\ldots,F_{j-1}}.$$
We want to show that it is pure of dimension $\dim(F_j)-1 = |V|-4$. The facets therein are the maximal elements in the set of all intersections $F_i\cap F_j$, where $i<j$. Clearly, $|F_i\cap F_j|\leq |V|-3$, since otherwise $F_i = F_i\cap F_j = F_j$. Suppose that $|F_i\cap F_j| \leq |V| - 4$. We will show that $F_i\cap F_j$ is not maximal. Indeed, we have that
$$|V|-4\geq |F_i\cap F_j| = |F_i|+|F_j| - |F_i\cup F_j| = 2|V|-4 - |F_i\cup F_j|,$$
which implies that $|F_i\cup F_j|\geq |V|$, whence $F_i\cup F_j = V$. By the definition of a weak shelling order, there is a $k$ with $i\ne k < j$ such that $F_i\cap F_j\subseteq F_k$. Say $\c{F_i} = \{v_i,w_i\}$, $\c{F_j} = \{v_j,w_j\}$ and $\c{F_k} = \{v_k,w_k\}$. Then $\{v_k,w_k\}\subseteq \{v_i,w_i,v_j,w_j\}$. Since the facets $F_i$ and $F_k$ are distinct either $v_k$ or $w_k$ is in $\{v_j,w_j\}$. This means that $|\c{F_k}\cup\c{F_j}|\leq 3$, that is, $|F_k\cap F_j|\geq |V|-3$. Hence $F_i\cap F_j$ is a proper subset of $F_k\cap F_j$, so it is not maximal.
\end{proof}

\end{document}